\documentclass{article}
\usepackage{amsfonts}
\usepackage{amsmath}

\newtheorem{theorem}{Theorem}

\newtheorem{corollary}[theorem]{Corollary}

\newtheorem{definition}[theorem]{Definition}

\newenvironment{proof}[1][Proof]{\noindent\textbf{#1.} }{\ \rule{0.5em}{0.5em}}
\input{tcilatex}
\begin{document}

\title{\textbf{Generating matrix of the bi-periodic Lucas numbers}}
\author{Arzu Coskun and Necati Taskara\thanks{%
e mail: \ \textit{arzucoskun58@gmail.com, ntaskara@selcuk.edu.tr}} \\
%EndAName
Department of Mathematics, Science Faculty, \\
Selcuk University, Campus, 42075, Konya, Turkey}
\maketitle

\begin{abstract}
In this paper, firstly, we introduce the $Q_{l}$-Generating matrix for the
bi-periodic Lucas numbers. Then, by taking into account this matrix
representation, we obtain some properties for the bi-periodic Fibonacci and
Lucas numbers.

\end{abstract}

\section{Introduction}

Fibonacci and Lucas numbers have attracted the attention of mathematicians
because of their intrinsic theory and applications \cite{11,15}. Though
closely related in definition, Lucas and Fibonacci numbers exhibit distinct
properties. After Fibonacci numbers firstly defined by Leonardo da Pisa at
the beginning of the thirteenth century, many authors have generalized this
sequence by using different methods \cite{4,6,9,16,18}.

Edson and Yayanie, in \cite{4}, defined bi-periodic Fibonacci sequence $%
\left\{ q_{n}\right\} _{n\in 
%TCIMACRO{\U{2115} }%
%BeginExpansion
\mathbb{N}
%EndExpansion
}$\ as%
\begin{equation}
q_{n}=\left\{ 
\begin{array}{c}
aq_{n-1}+q_{n-2},\text{ }n\text{ }even \\ 
bq_{n-1}+q_{n-2},\text{ }n\text{ }odd%
\end{array}%
\right.  \label{1.1}
\end{equation}%
where $q_{0}=0,$ $q_{1}=1$ and $a$,$b$ are nonzero real numbers.

Also, the Binet formula of generalized Fibonacci sequence is given by%
\begin{equation}
q_{n}=\frac{1}{a^{\left\lfloor \frac{n-1}{2}\right\rfloor }b^{\left\lfloor 
\frac{n}{2}\right\rfloor }}\left( \frac{\alpha ^{n}-\beta ^{n}}{\alpha
-\beta }\right)  \label{1.2}
\end{equation}%
where $\alpha =\frac{ab+\sqrt{a^{2}b^{2}+4ab}}{2}$, $\beta =\frac{ab-\sqrt{%
a^{2}b^{2}+4ab}}{2}$ and $\varepsilon (n)=n-2\left\lfloor \frac{n}{2}%
\right\rfloor .$

Cassini identities for bi-periodic Fibonacci sequence is defined by the form%
\begin{equation}
a^{1-\varepsilon (n)}b^{\varepsilon (n)}q_{n-1}q_{n+1}-a^{\varepsilon
(n)}b^{1-\varepsilon (n)}q_{n}^{2}=a(-1)^{n}.  \label{1.3}
\end{equation}

On the other hand, the Lucas numbers are an integer sequence named after the
mathematician Fran\c{c}ois \'{E}douard Anatole Lucas (1842--91), who studied
both that sequence and the closely related Fibonacci numbers. Bilgici, in 
\cite{2},defined a generalization of Lucas sequence by the recurrence
relation%
\begin{equation}
l_{n}=\left\{ 
\begin{array}{c}
bl_{n-1}+l_{n-2},\text{ }n\text{ }even \\ 
al_{n-1}+l_{n-2},\text{ }n\text{ }odd%
\end{array}%
\right.  \label{1.4}
\end{equation}%
where $l_{0}=2,$ $l_{1}=a$ and $a$,$b$ are nonzero real numbers. The Binet
formula of this sequence is%
\begin{equation}
l_{n}=\frac{1}{a^{\left\lfloor \frac{n}{2}\right\rfloor }b^{\left\lfloor 
\frac{n+1}{2}\right\rfloor }}\left( \alpha ^{n}+\beta ^{n}\right) .
\label{1.5}
\end{equation}

And Cassini identities for bi-periodic Lucas sequence is given by%
\begin{equation}
\left( \frac{b}{a}\right) ^{\varepsilon (n+1)}l_{n-1}l_{n+1}-\left( \frac{b}{%
a}\right) ^{\varepsilon (n)}l_{n}^{2}=(-1)^{n+1}\left( ab+4\right) .
\label{1.6}
\end{equation}

Additionaly, some authors studied matrix representation of special number
sequences \cite{1,3,5,7,8,10},\cite{12}-\cite{14},\cite{17,19}. In one of
these studies, Sylvester \cite{12} gave Fibonacci $Q$-matrix as%
\begin{equation*}
Q=\left[ 
\begin{array}{cc}
1 & 1 \\ 
1 & 0%
\end{array}%
\right] .
\end{equation*}

Then, for a pozitive integer $n$, the author found that $Q^{n}$ has the form%
\begin{equation*}
Q^{n}=\left[ 
\begin{array}{cc}
F_{n+1} & F_{n} \\ 
F_{n} & F_{n-1}%
\end{array}%
\right] .
\end{equation*}

Also, by using Fibonacci $Q$-matrix, Cassini's Fibonacci formula can be
found as follows:%
\begin{equation*}
F_{n+1}F_{n-1}-F_{n}^{2}=(-1)^{n}.
\end{equation*}

Similarly, in \cite{10}, Koken and Bozkurt defined the Lucas $Q_{L}$-matrix
as%
\begin{equation*}
Q_{L}=\left[ 
\begin{array}{cc}
3 & 1 \\ 
1 & 2%
\end{array}%
\right] .
\end{equation*}

Also, they found some well-known equalities and a Binet-like formula for the
Lucas numbers.

\bigskip

In the light of the above studies, in here, we define \ the $Q_{l}$%
-Generating matrix for bi-periodic Lucas numbers. This matrix is
generalization form of well-known Lucas $Q_{L}$-matrix. After, by using this
matrix representation, we have found some equalities for the bi-periodic
Fibonacci and Lucas numbers.

\section{The properties of $Q_{l}$-Generating matrix}

In this section we firstly introduce a generalized matrix representation of
the bi-periodic Lucas numbers. Then, we investigate the $n$th power,
determinant and inverse of $Q_{l}$-Generating matrix.

\begin{definition}
Bi-periodic Lucas $Q_{l}$-Generating matrix define by%
\begin{equation}
Q_{l}=\left[ 
\begin{array}{cc}
a^{2}+2\frac{a}{b} & \frac{a^{2}}{b} \\ 
a & 2\frac{a}{b}%
\end{array}%
\right] .  \label{2.1}
\end{equation}
\end{definition}

\begin{theorem}
Let $Q_{l}$-Generating matrix be as in (\ref{2.1}).Then, for every $n\in 
%TCIMACRO{\U{2124} }%
%BeginExpansion
\mathbb{Z}
%EndExpansion
^{+}$, we have 
\begin{equation}
Q_{l}^{n}=\left\{ 
\begin{array}{c}
\left( \frac{a}{b}\right) ^{n}\left( ab+4\right) ^{\frac{n}{2}}\left[ 
\begin{array}{cc}
q_{n+1} & q_{n} \\ 
\frac{b}{a}q_{n} & q_{n-1}%
\end{array}%
\right] ,\text{ }n\text{ }even \\ 
\left( \frac{a}{b}\right) ^{n}\left( ab+4\right) ^{\frac{n-1}{2}}\left[ 
\begin{array}{cc}
l_{n+1} & l_{n} \\ 
\frac{b}{a}l_{n} & l_{n-1}%
\end{array}%
\right] ,\text{ }n\text{ }odd%
\end{array}%
\right. ,  \label{2.2}
\end{equation}%
where $q_{n},l_{n}$ are the $n$th bi-periodic Fibonacci and Lucas numbers,
respectively.

\begin{proof}
We use mathematical induction on $n$. Since $l_{0}=2,$ $l_{1}=a,$ $%
l_{2}=ab+2,$ $q_{1}=1,$ $q_{2}=a$ and $q_{3}=ab+1$ we write 
\begin{eqnarray*}
Q_{l} &=&\left[ 
\begin{array}{cc}
a^{2}+2\frac{a}{b} & \frac{a^{2}}{b} \\ 
a & 2\frac{a}{b}%
\end{array}%
\right] \\
&=&\left( \frac{a}{b}\right) \left( ab+4\right) ^{0}\left[ 
\begin{array}{cc}
l_{2} & l_{1} \\ 
\frac{b}{a}l_{1} & l_{0}%
\end{array}%
\right]
\end{eqnarray*}%
\begin{eqnarray*}
Q_{l}^{2} &=&\left[ 
\begin{array}{cc}
a^{4}+5\frac{a^{3}}{b}+4\frac{a^{2}}{b^{2}} & \frac{a^{4}}{b}+4\frac{a^{3}}{%
b^{2}} \\ 
a^{3}+4\frac{a^{2}}{b} & \frac{a^{3}}{b}+4\frac{a^{2}}{b^{2}}%
\end{array}%
\right] \\
&=&\left( \frac{a}{b}\right) ^{2}\left( ab+4\right) \left[ 
\begin{array}{cc}
q_{3} & q_{2} \\ 
\frac{b}{a}q_{2} & q_{1}%
\end{array}%
\right]
\end{eqnarray*}%
which show that the equation (\ref{2.2}) is true for $n=1$ and $n=2$. Now we
suppose that it is true for $n=k,$ that is,%
\begin{equation*}
Q_{l}^{k}=\left\{ 
\begin{array}{c}
\left( \frac{a}{b}\right) ^{k}\left( ab+4\right) ^{\frac{k}{2}}\left[ 
\begin{array}{cc}
q_{k+1} & q_{k} \\ 
\frac{b}{a}q_{k} & q_{k-1}%
\end{array}%
\right] ,\text{ }k\text{ }even \\ 
\left( \frac{a}{b}\right) ^{k}\left( ab+4\right) ^{\frac{k-1}{2}}\left[ 
\begin{array}{cc}
l_{k+1} & l_{k} \\ 
\frac{b}{a}l_{k} & l_{k-1}%
\end{array}%
\right] ,\text{ }k\text{ }odd%
\end{array}%
\right. .
\end{equation*}%
If we supposed that $k$ is even, by using properties of the bi-periodic
Fibonacci numbers, we obtain%
\begin{eqnarray*}
Q_{l}^{k+2} &=&Q_{l}^{k}Q_{l}^{2} \\
&=&\left( \frac{a}{b}\right) ^{k+2}\left( ab+4\right) ^{\frac{k}{2}+1}\left[ 
\begin{array}{cc}
q_{k+3} & q_{k+2} \\ 
\frac{b}{a}q_{k+2} & q_{k+1}%
\end{array}%
\right] .
\end{eqnarray*}%
Similarly, for $k$ is odd, we can write%
\begin{eqnarray*}
Q_{l}^{k+2} &=&Q_{l}^{k}Q_{l}^{2} \\
&=&\left( \frac{a}{b}\right) ^{k+2}\left( ab+4\right) ^{\frac{k+1}{2}}\left[ 
\begin{array}{cc}
l_{k+3} & l_{k+2} \\ 
\frac{b}{a}l_{k+2} & l_{k+1}%
\end{array}%
\right]
\end{eqnarray*}%
And if we compose this partial function, we conclude%
\begin{equation*}
Q_{l}^{k+2}=\left\{ 
\begin{array}{c}
\left( \frac{a}{b}\right) ^{k+2}\left( ab+4\right) ^{\frac{k+2}{2}}\left[ 
\begin{array}{cc}
q_{k+3} & q_{k+2} \\ 
\frac{b}{a}q_{k+2} & q_{k+1}%
\end{array}%
\right] ,\text{ }k+2\text{ }even \\ 
\left( \frac{a}{b}\right) ^{k+2}\left( ab+4\right) ^{\frac{k+1}{2}}\left[ 
\begin{array}{cc}
l_{k+3} & l_{k+2} \\ 
\frac{b}{a}l_{k+2} & l_{k+1}%
\end{array}%
\right] ,\text{ }k+2\text{ }odd%
\end{array}%
\right. .
\end{equation*}%
which is desired. Hence the proof is completed.
\end{proof}
\end{theorem}

By using above theorem, we can give the following corollary.

\begin{corollary}
Let $Q_{l}^{n}$ be as in (\ref{2.2}). Then the following equality is valid
for all positive integers:
\end{corollary}

\begin{equation*}
\det (Q_{l}^{n})=\left( \frac{a^{2}}{b^{2}}\left( ab+4\right) \right) ^{n}.
\end{equation*}

\begin{proof}
If we prove using iteration, then we obtain%
\begin{eqnarray*}
\det (Q_{l}) &=&\left\vert 
\begin{array}{cc}
a^{2}+2\frac{a}{b} & \frac{a^{2}}{b} \\ 
a & 2\frac{a}{b}%
\end{array}%
\right\vert =\frac{a^{3}}{b}+4\frac{a^{2}}{b^{2}}=\frac{a^{2}}{b^{2}}\left(
ab+4\right) , \\
\det (Q_{l}^{2}) &=&\left\vert 
\begin{array}{cc}
a^{4}+3\frac{a^{3}}{b}+4\frac{a^{2}}{b^{2}} & \frac{a^{2}}{b} \\ 
a & 2\frac{a}{b}%
\end{array}%
\right\vert =\left( \frac{a^{2}}{b^{2}}\left( ab+4\right) \right) ^{2}, \\
&&\vdots  \\
\det \left( Q_{l}^{n}\right)  &=&\left\{ 
\begin{array}{c}
\left( \frac{a}{b}\right) ^{2n}\left( ab+4\right) ^{n}\left( q_{n+1}q_{n-1}-%
\frac{b}{a}q_{n}^{2}\right) ,\text{ }n\text{ }even \\ 
\left( \frac{a}{b}\right) ^{2n}\left( ab+4\right) ^{n-1}\left(
l_{n+1}l_{n-1}-\frac{b}{a}l_{n}^{2}\right) ,\text{ }n\text{ }odd%
\end{array}%
\right.  \\
&=&\left( \frac{a}{b}\right) ^{2n}\left( ab+4\right) ^{n}.
\end{eqnarray*}
\end{proof}

\bigskip 

Cassini and Binet formulas for bi-periodic Fibonacci and Lucas sequences are
given in \cite{2,4}. In here, with a new perspective, we rewrite these
properties using $Q_{l}$-Generating matrix.

\begin{theorem}
The following equalities are valid for all positive integers:
\end{theorem}

\begin{itemize}
\item[$i)$] $a^{1-\varepsilon (n)}b^{\varepsilon
(n)}q_{n+1}q_{n-1}-a^{1-\varepsilon (n)}b^{\varepsilon (n)}q_{n}^{2}=a\left(
-1\right) ^{n},$

\item[$ii)$] $\left( \frac{b}{a}\right) ^{\varepsilon
(n+1)}l_{n-1}l_{n+1}-\left( \frac{b}{a}\right) ^{\varepsilon
(n)}l_{n}^{2}=\left( ab+4\right) \left( -1\right) ^{n+1}.$
\end{itemize}

\begin{proof}
By using Theorem 2 and Corollary 3, we obtain for even and odd $n$%
\begin{equation*}
q_{n+1}q_{n-1}-\frac{b}{a}q_{n}^{2}=1,
\end{equation*}%
\begin{equation*}
l_{n-1}l_{n+1}-\frac{b}{a}l_{n}^{2}=\left( ab+4\right) ,
\end{equation*}%
respectively. That is, for $n\in 
%TCIMACRO{\U{2124} }%
%BeginExpansion
\mathbb{Z}
%EndExpansion
^{+},$%
\begin{equation*}
aq_{2n+1}q_{2n-1}-bq_{2n}^{2}=a,
\end{equation*}%
\begin{equation*}
al_{2n}l_{2n+2}-bl_{2n+1}^{2}=a\left( ab+4\right) .
\end{equation*}%
Then, by using $q_{2n-1}=bq_{2n}-q_{2n+1}$ and $l_{2n+2}=bl_{2n+1}+l_{2n}$,\
we obtain%
\begin{equation*}
bq_{2n+2}q_{2n}-aq_{2n+1}^{2}=a\left( -1\right) ,
\end{equation*}%
\begin{equation*}
bl_{2n-1}l_{2n+1}-al_{2n}^{2}=a\left( ab+4\right) \left( -1\right) .
\end{equation*}%
If we compose these equations, we conclude%
\begin{equation*}
a^{1-\varepsilon (n)}b^{\varepsilon (n)}q_{n+1}q_{n-1}-a^{1-\varepsilon
(n)}b^{\varepsilon (n)}q_{n}^{2}=a\left( -1\right) ^{n},
\end{equation*}%
\begin{equation*}
\left( \frac{b}{a}\right) ^{\varepsilon (n+1)}l_{n-1}l_{n+1}-\left( \frac{b}{%
a}\right) ^{\varepsilon (n)}l_{n}^{2}=\left( ab+4\right) \left( -1\right)
^{n+1}.
\end{equation*}%
\bigskip 
\end{proof}

\begin{theorem}
Let $n$ be any integer. The Binet formulas of bi-periodic Fibonacci and
Lucas numbers are%
\begin{equation*}
q_{n}=\left( \frac{a^{1-\varepsilon (n)}}{\left( ab\right) ^{\left\lfloor 
\frac{n}{2}\right\rfloor }}\right) \frac{\alpha ^{n}-\beta ^{n}}{\alpha
-\beta },
\end{equation*}%
\begin{equation*}
l_{n}=\frac{1}{a^{\left\lfloor \frac{n}{2}\right\rfloor }b^{\left\lfloor 
\frac{n+1}{2}\right\rfloor }}\left( \alpha ^{n}+\beta ^{n}\right) ,
\end{equation*}%
where $\alpha $ and $\beta $ are roots of $X^{2}-abX-ab=0$ equation.
\end{theorem}

\begin{proof}
Let the matrix $Q_{l}$ be as in (\ref{2.1}). Characteristic equation of $%
Q_{l}$-Generating matrix is%
\begin{equation*}
\lambda ^{2}-\left( a^{2}+4\frac{a}{b}\right) \lambda +\frac{a^{3}}{b}+4%
\frac{a^{2}}{b^{2}}=0.
\end{equation*}%
Then, eigenvalues and eigenvectors of the matrix $Q_{l}$ are%
\begin{equation*}
\lambda _{1}=\left( \frac{a^{\frac{1}{2}}}{b^{\frac{3}{2}}}\right) \left(
ab+4\right) ^{\frac{1}{2}}\alpha ,\lambda _{2}=\left( \frac{a^{\frac{1}{2}}}{%
b^{\frac{3}{2}}}\right) \left( ab+4\right) ^{\frac{1}{2}}\left( -\beta
\right) 
\end{equation*}%
and%
\begin{equation*}
u_{1}=\left( \frac{a^{2}}{b},-\frac{a}{b}\beta \right) ,u_{2}=\left( \frac{%
a^{2}}{b},-\frac{a}{b}\alpha \right) 
\end{equation*}%
where $\alpha =\frac{ab+\sqrt{a^{2}b^{2}+4ab}}{2}$ and $\beta =\frac{ab-%
\sqrt{a^{2}b^{2}+4ab}}{2}$. $Q_{l}$-Generating matrix can be diagonalized by
using%
\begin{equation*}
V=U^{-1}Q_{l}U,
\end{equation*}%
which%
\begin{equation*}
U=\left( 
\begin{array}{cc}
u_{1}^{T} & u_{2}^{T}%
\end{array}%
\right) =\left[ 
\begin{array}{cc}
\frac{a^{2}}{b} & \frac{a^{2}}{b} \\ 
-\frac{a}{b}\beta  & -\frac{a}{b}\alpha 
\end{array}%
\right] 
\end{equation*}%
and%
\begin{eqnarray*}
V &=&diag\left( \lambda _{1},\lambda _{2}\right)  \\
&=&\left[ 
\begin{array}{cc}
\left( \frac{a^{\frac{1}{2}}}{b^{\frac{3}{2}}}\right) \left( ab+4\right) ^{%
\frac{1}{2}}\alpha  & 0 \\ 
0 & \left( \frac{a^{\frac{1}{2}}}{b^{\frac{3}{2}}}\right) \left( ab+4\right)
^{\frac{1}{2}}\left( -\beta \right) 
\end{array}%
\right] .
\end{eqnarray*}%
From properties of similar matrices, for $n$ is any integer, we obtain%
\begin{equation*}
Q_{l}^{n}=UV^{n}U^{-1}.
\end{equation*}%
Thus, we get%
\begin{eqnarray*}
Q_{l}^{n} &=&\left( \frac{a^{\frac{n}{2}}}{b^{\frac{3n}{2}}}\right) \frac{%
\left( ab+4\right) ^{\frac{n}{2}}}{\alpha -\beta }\left[ 
\begin{array}{cc}
\alpha ^{n+1}-\beta \left( -\beta \right) ^{n} & a\alpha ^{n}-a\left( -\beta
\right) ^{n} \\ 
b\alpha ^{n}-b\left( -\beta \right) ^{n} & -\beta \alpha ^{n}+\alpha \left(
-\beta \right) ^{n}%
\end{array}%
\right]  \\
&=&\left( \frac{a^{\frac{n-1}{2}}}{b^{\frac{3n+1}{2}}}\right) \left(
ab+4\right) ^{\frac{n-1}{2}}\left[ 
\begin{array}{cc}
\alpha ^{n+1}-\beta \left( -\beta \right) ^{n} & a\alpha ^{n}-a\left( -\beta
\right) ^{n} \\ 
b\alpha ^{n}-b\left( -\beta \right) ^{n} & -\beta \alpha ^{n}+\alpha \left(
-\beta \right) ^{n}%
\end{array}%
\right] .
\end{eqnarray*}%
Taking into account the Theorem 2, for the case $n$ is even and odd, we can
write%
\begin{equation*}
q_{n}=\left( \frac{a}{\left( ab\right) ^{\frac{n}{2}}}\right) \frac{\alpha
^{n}-\beta ^{n}}{\alpha -\beta },
\end{equation*}%
\begin{equation*}
l_{n}=\frac{1}{a^{\frac{n-1}{2}}b^{\frac{n}{2}}}\left( \alpha ^{n}+\beta
^{n}\right) .
\end{equation*}%
respectively. That is,%
\begin{equation*}
q_{2n}=\left( \frac{a}{\left( ab\right) ^{n}}\right) \frac{\alpha
^{2n}-\beta ^{2n}}{\alpha -\beta },
\end{equation*}%
\begin{equation*}
l_{2n+1}=\frac{1}{a^{n}b^{n+1}}\left( \alpha ^{2n+1}+\beta ^{2n+1}\right) .
\end{equation*}%
And also, since $q_{2n+2}=aq_{2n+1}+q_{2n}$ and $l_{2n+1}=al_{2n}+l_{2n-1},$
we have%
\begin{equation*}
q_{2n+1}=\left( \frac{1}{\left( ab\right) ^{n}}\right) \frac{\alpha
^{2n+1}-\beta ^{2n+1}}{\alpha -\beta },
\end{equation*}%
\begin{equation*}
l_{2n}=\frac{1}{\left( ab\right) ^{n}}\left( \alpha ^{2n}+\beta ^{2n}\right)
.
\end{equation*}%
If we compose the obtained results, the desired result is obtained.
\end{proof}

In following theorem, we offer relationships between bi-periodic Fibonacci
sequence and bi-periodic Lucas sequence.

\begin{theorem}
For $m,n\in 
%TCIMACRO{\U{2124} }%
%BeginExpansion
\mathbb{Z}
%EndExpansion
$, the following statements are true:

\begin{itemize}
\item[$i)$] $\left( ab+4\right) q_{2\left( m+n+1\right) }=l_{2m+1}l_{2\left(
n+1\right) }+l_{2m}l_{2n+1}$,

\item[$ii)$] $q_{2\left( m+n\right) }=q_{2m}q_{2n+1}+q_{2m-1}q_{2n}$,

\item[$iii)$] $l_{2\left( m+n\right) +1}=l_{2m+1}q_{2n+1}+l_{2m}q_{2n}$,

\item[$iv)$] $\left( ab+4\right) q_{2\left( m-n\right) }=l_{2m+1}l_{2\left(
n+1\right) }-l_{2\left( m+1\right) }l_{2n+1}$,

\item[$v)$] $q_{2\left( m-n\right) }=q_{2m}q_{2n+1}-q_{2m+1}q_{2n}$,

\item[$vi)$] $l_{2\left( m-n\right) +1}=q_{2m+1}l_{2n+1}-q_{2\left(
m+1\right) }l_{2n}$.
\end{itemize}
\end{theorem}

\begin{proof}
Using (\ref{2.2}), $Q_{l}^{m+n}$ can be written as%
\begin{equation}
Q_{l}^{m+n}=\left\{ 
\begin{array}{c}
\left( \frac{a}{b}\right) ^{m+n}\left( ab+4\right) ^{\frac{m+n}{2}}\left[ 
\begin{array}{cc}
q_{m+n+1} & q_{m+n} \\ 
\frac{b}{a}q_{m+n} & q_{m+n-1}%
\end{array}%
\right] ,\text{ }m+n\text{ }even \\ 
\left( \frac{a}{b}\right) ^{m+n}\left( ab+4\right) ^{\frac{m+n-1}{2}}\left[ 
\begin{array}{cc}
l_{m+n+1} & l_{m+n} \\ 
\frac{b}{a}l_{m+n} & l_{m+n-1}%
\end{array}%
\right] ,\text{ }m+n\text{ }odd%
\end{array}%
\right. .  \label{2.3}
\end{equation}%
For the case of odd $m$ and $n$, we can write%
\begin{equation}
Q_{l}^{m}Q_{l}^{n}=\left( \frac{a}{b}\right) ^{m+n}\left( ab+4\right) ^{%
\frac{m+n}{2}-1}\left[ 
\begin{array}{cc}
l_{m+1}l_{n+1}+\frac{b}{a}l_{m}l_{n} & l_{m+1}l_{n}+l_{m}l_{n-1} \\ 
\frac{b}{a}\left\{ l_{m}l_{n+1}+l_{m-1}l_{n}\right\} & \frac{b}{a}%
l_{m}l_{n}+l_{m-1}l_{n-1}%
\end{array}%
\right] .  \label{2.4}
\end{equation}%
If we compare the $1$\textit{st} row and $2$\textit{nd} column entries of
the matrices (\ref{2.3}) and (\ref{2.4}), we get%
\begin{equation*}
\left( ab+4\right) q_{m+n}=l_{m+1}l_{n}+l_{m}l_{n-1},
\end{equation*}%
On the other hand, comparing the entries $2$\textit{nd} row and $1$\textit{st%
} column, we obtain%
\begin{equation*}
\left( ab+4\right) q_{m+n}=l_{m}l_{n+1}+l_{m-1}l_{n}.
\end{equation*}%
For the case of even $m$ and $n$,%
\begin{equation}
Q_{l}^{m}Q_{l}^{n}=\left( \frac{a}{b}\right) ^{m+n}\left( ab+4\right) ^{%
\frac{m+n}{2}}\left[ 
\begin{array}{cc}
q_{m+1}q_{n+1}+\frac{b}{a}q_{m}q_{n} & q_{m+1}q_{n}+q_{m}q_{n-1} \\ 
\frac{b}{a}\left\{ q_{m}q_{n+1}+q_{m-1}q_{n}\right\} & \frac{b}{a}%
q_{m}q_{n}+q_{m-1}q_{n-1}%
\end{array}%
\right] .  \label{2.5}
\end{equation}%
Comparing the entries (1,2) and (2,1) for the matrices (\ref{2.3}) and (\ref%
{2.5}), we find%
\begin{equation*}
q_{m+n}=q_{m+1}q_{n}+q_{m}q_{n-1},
\end{equation*}%
\begin{equation*}
q_{m+n}=q_{m}q_{n+1}+q_{m-1}q_{n}.
\end{equation*}%
For the case of odd $m$ and even $n$ (or case of even $m$ and odd $n)$,%
\begin{equation}
Q_{l}^{m}Q_{l}^{n}=\left( \frac{a}{b}\right) ^{m+n}\left( ab+4\right) ^{%
\frac{m+n-1}{2}}\left[ 
\begin{array}{cc}
l_{m+1}q_{n+1}+\frac{b}{a}l_{m}q_{n} & l_{m+1}q_{n}+l_{m}q_{n-1} \\ 
\frac{b}{a}\left\{ l_{m}q_{n+1}+l_{m-1}q_{n}\right\} & \frac{b}{a}%
l_{m}q_{n}+l_{m-1}q_{n-1}%
\end{array}%
\right] .  \label{2.6}
\end{equation}%
And if we compare the entries (1,2) and (2,1) for the matrices (\ref{2.3})
and (\ref{2.6}), we acquire%
\begin{equation*}
l_{m+n}=l_{m+1}q_{n}+l_{m}q_{n-1},
\end{equation*}%
\begin{equation*}
l_{m+n}=l_{m}q_{n+1}+l_{m-1}q_{n}.
\end{equation*}%
Similarly, by calculating inverse of the matrix $Q_{l}^{n}$ in (\ref{2.2}),
we conclude%
\begin{equation*}
Q_{l}^{-n}=\left\{ 
\begin{array}{c}
\left( \frac{a}{b}\right) ^{-n}\left( ab+4\right) ^{-\frac{n}{2}}\left[ 
\begin{array}{cc}
q_{n-1} & -q_{n} \\ 
-\frac{b}{a}q_{n} & q_{n+1}%
\end{array}%
\right] ,\text{ }n\text{ }even \\ 
\left( \frac{a}{b}\right) ^{-n}\left( ab+4\right) ^{-\frac{n+1}{2}}\left[ 
\begin{array}{cc}
l_{n-1} & -l_{n} \\ 
-\frac{b}{a}l_{n} & l_{n11}%
\end{array}%
\right] ,\text{ }n\text{ }odd%
\end{array}%
\right. .
\end{equation*}%
Benefiting from the equality $Q_{l}^{m-n}=Q_{l}^{m}Q_{l}^{-n}$ and by
comparing the entries (1,2) and (2,1) of these matrices, the desired result
can be obtained. That is, for the case of odd $m$\ and $n$, we get%
\begin{equation*}
\left( ab+4\right) q_{m-n}=l_{m}l_{n+1}-l_{m+1}l_{n}.
\end{equation*}%
For the case of even $m$ and $n$, we obtain%
\begin{equation*}
q_{m-n}=q_{m}q_{n+1}-q_{m+1}q_{n}.
\end{equation*}%
Finally, for the case of even $m$ and odd $n$ (or case of odd $m$ and even $%
n $), we acquire%
\begin{equation*}
l_{m-n}=q_{m}l_{n+1}-q_{m+1}l_{n},
\end{equation*}%
Thus, we have the desired expressions.
\end{proof}

\section{Conclusion}

This paper presents the some properties of bi-periodic Fibonacci and Lucas
numbers and relationships between these sequences by using the $Q_{l}$%
-Generating matrix. Also, some well-known matrices are special cases of this
generating matrix. For example, if we choose $a=b=1\ $and $a=b=k$, we get
the Lucas $Q_{L}$-matrix and $k$-Lucas Companion matrix, respectively.

\section{Acknowledgement}

This study is a part of Arzu Co\c{s}kun's Ph.D. Thesis. Thank to the editor
and reviewers for their interests and valuable comments.

\end{document}